\documentclass[11pt]{article}
\usepackage{amsmath,amssymb,amscd,amsthm}
\usepackage{xcolor}
\usepackage{multirow}
\newdimen\unit\newdimen\psep\newcount\nd\newcount\ndx\newbox\dotb\newbox\ptbox
\newdimen\dx\newdimen\dy\newdimen\dxx\newdimen\dyy\newdimen\hgt
\newdimen\xoff\newdimen\yoff
\newcommand\clap[1]{\hbox to 0pt{\hss{#1}\hss}}
\newcommand\vdisk[1]{{\font\dotf=cmr10 scaled #1\dotf.}}
\newcommand\varline[2]{\setbox\dotb\hbox{\vdisk{#1}}\xoff=-.5\wd\dotb
\wd\dotb=0pt\yoff=-.5\ht\dotb\psep=#2\ht\dotb}
\newcommand\varpt[1]{\setbox\ptbox\clap{\vdisk{#1}}\setbox\ptbox
\hbox{\raise-.5\ht\ptbox\box\ptbox}}
\newcommand\cpt{\copy\ptbox}
\newcommand\point[3]{\rlap{\kern#1\unit\raise#2\unit\hbox{#3}}}
\newcommand\setnd[4]{\dx=#3\unit\advance\dx-#1\unit\divide\dx by\psep
\dy=#4\unit\advance\dy-#2\unit\divide\dy by\psep \multiply\dx
by\dx\multiply\dy by\dy\advance\dx\dy\nd=1\advance\dx-1sp
\loop\ifnum\dx>0\advance\dx-\nd sp\advance\nd1\advance\dx-\nd
sp\repeat}
\newcommand\dl[4]{{\setnd{#1}{#2}{#3}{#4}\dline{#1}{#2}{#3}{#4}\nd}}
\newcommand\dline[5]{{\nd=#5\hgt=#2\unit\dx=#3\unit\advance\dx-#1\unit
\divide\dx by\nd\dy=#4\unit\advance\dy-#2\unit\divide\dy by\nd
\advance\hgt\yoff\rlap{\kern#1\unit\kern\xoff\loop\ifnum\nd>1\advance\nd-1
\advance\hgt\dy\kern\dx\raise\hgt\copy\dotb\repeat}}}
\newcommand\ellipse[4]{\qellip{#1}{#2}{#3}{#4}\qellip{#1}{#2}{#3}{-#4}%
\qellip{#1}{#2}{-#3}{#4}\qellip{#1}{#2}{-#3}{-#4}}
\newcommand\qellip[4]{{\setnd{0}{0}{#3}{#4}\dx=\unit\dy=0pt\raise\yoff\rlap{%
\kern#1\unit\kern\xoff\raise#2\unit\hbox{\loop\ifnum\dx>0\rlap{\kern#3\dx
\raise#4\dy\copy\dotb}\hgt=\dx\divide\hgt
by\nd\advance\dy\hgt\hgt=\dy \divide\hgt
by\nd\advance\dx-\hgt\repeat\rlap{\raise#4\dy\copy\dotb}}}}}
\newcommand\bez[6]{{\setnd{#1}{#2}{#3}{#4}\ndx=\nd\setnd{#3}{#4}{#5}{#6}
\ifnum\ndx>\nd\nd=\ndx\fi\dx=#3\unit\advance\dx-#1\unit\dy=#4\unit
\advance\dy-#2\unit\dxx=#5\unit\advance\dxx-#1\unit\dyy=#6\unit\advance
\dyy-#2\unit\advance\dxx-2\dx\advance\dyy-2\dy\divide\dxx
by\nd\divide\dyy
by\nd\advance\dx.25\dxx\advance\dy.25\dyy\divide\dx
by\nd\divide\dy by\nd \multiply\nd
by2\dx=100\dx\dy=100\dy\dxx=100\dxx\dyy=100\dyy\divide\dxx by\nd
\divide\dyy
by\nd\hgt=#2\unit\raise\yoff\rlap{\kern#1\unit\kern\xoff
\raise\hgt\copy\dotb\loop\ifnum\nd>0\advance\nd-1\advance\hgt0.01\dy
\kern0.01\dx\raise\hgt\copy\dotb\advance\dx\dxx\advance\dy\dyy\repeat}}}
\newcommand\ptu[3]{\point{#1}{#2}{\cpt\raise1ex\clap{$\scriptstyle{#3}$}}}
\newcommand\ptd[3]{\point{#1}{#2}{\cpt\raise-1.8ex\clap{$\scriptstyle{#3}$}}}
\newcommand\ptr[3]{\point{#1}{#2}{\cpt\raise-.4ex\rlap{$\ \scriptstyle{#3}$}}}
\newcommand\ptl[3]{\point{#1}{#2}{\cpt\raise-.4ex\llap{$\scriptstyle{#3}\ $}}}
\newcommand\ptlu[3]{\point{#1}{#2}{\raise.8ex\clap{$\scriptstyle{#3}$}}}
\newcommand\ptld[3]{\point{#1}{#2}{\raise-1.6ex\clap{$\scriptstyle{#3}$}}}
\newcommand\ptlr[3]{\point{#1}{#2}{\raise-.4ex\rlap{$\,\scriptstyle{#3}$}}}
\newcommand\ptll[3]{\point{#1}{#2}{\raise-.4ex\llap{$\scriptstyle{#3}\,$}}}
\newcommand\pt[2]{\point{#1}{#2}{\cpt}}

\newcommand\thnline{\varline{400}{.6}}
\newcommand\dotline{\varline{1000}{4}}
\varpt{2500}\thnline\unit=2em

\newtheorem{thm}{Theorem}

\newtheorem{prop}[thm]{Proposition}

\newtheorem{lem}[thm]{Lemma}

\newtheorem{clm}[thm]{Claim}

\newtheorem{pro}[thm]{Problem}
\newtheorem*{pro1a}{Problem 1$'$}
\newtheorem*{pro2a}{Problem 2$'$}
\newtheorem*{rmk}{Remark}

\def\diam{\textup{diam}}

\begin{document}

\title{The size of graphs with restricted\\
rainbow $2$-connection number}

\author{Shinya Fujita\\[1.5ex]
\small School of Data Science, Yokohama City University,\\
\small Yokohama 236-0027, Japan\\
\small \tt fujita@yokohama-cu.ac.jp\\
[2.5ex]
Henry Liu\\[1.5ex]
\small School of Mathematics, Sun Yat-sen University,\\
\small Guangzhou 510275, China\\
\small\tt liaozhx5@mail.sysu.edu.cn\\
[2.5ex]
Boram Park\\[1.5ex]
\small Department of Mathematics, Ajou University,\\
\small Suwon 16499, Republic of Korea \\
\small\tt borampark@ajou.ac.kr
}


\date{4 September 2020}

\maketitle
\makeatother

\begin{abstract}
Let $k$ be a positive integer, and $G$ be a $k$-connected graph. An edge-coloured path is \emph{rainbow} if all of its edges have distinct colours. The \emph{rainbow $k$-connection number} of $G$, denoted by $rc_k(G)$, is the minimum number of colours in an edge-colouring of $G$ such that, any two vertices are connected by $k$ internally vertex-disjoint rainbow paths. The function $rc_k(G)$ was introduced by Chartrand, Johns, McKeon and Zhang in 2009, and has since attracted significant interest. Let $t_k(n,r)$ denote the minimum number of edges in a $k$-connected graph $G$ on $n$ vertices with $rc_k(G)\le r$. Let $s_k(n,r)$ denote the maximum number of edges in a $k$-connected graph $G$ on $n$ vertices with $rc_k(G)\ge r$. The functions $t_1(n,r)$ and $s_1(n,r)$ have previously been studied by various authors. In this paper, we study the functions $t_2(n,r)$ and $s_2(n,r)$. We determine bounds for $t_2(n,r)$ which imply that $t_2(n,2)=(1+o(1))n\log_2 n$, and $t_2(n,r)$ is linear in $n$ for $r\ge 3$. We also provide some remarks about the function $s_2(n,r)$.\\

\noindent\textbf{AMS Subject Classificiation (2010):} 05C15, 05C35, 05C40\\

\end{abstract}


\section{Introduction}\label{Intro}

All graphs in this paper are finite, simple and undirected. For a graph $G$ and vertices $x,y\in V(G)$, let $d_G(x,y)$ denote the distance (i.e., the length of a shortest path) from $x$ to $y$ in $G$, and let  $\diam(G)=\max\{d_G(x,y):x,y\in V(G)\}$ denote the diameter of $G$. Let $\deg_G(x)$ denote the degree of $x$ in $G$. For $X\subset V(G)$, let $G[X]$ denote the subgraph of $G$ induced by $X$. For disjoint subsets $X,Y\subset V(G)$, let $E_G(X,Y)$ denote the set of edges of $G$ with one end-vertex in $X$ and the other in $Y$. For a path $P$ and $x,y\in V(P)$ (possibly $x=y$), we write $xPy$ for the subpath of $P$ with end-vertices $x$ and $y$. A path $P=x_0x_1\cdots x_\ell$, for some $\ell\ge 1$, is called an \emph{$\ell$-ear} of $G$, or simply an \emph{ear}, if $V(P)\cap V(G)=\{x_0,x_\ell\}$ and $E(P)\cap E(G)=\emptyset$. For $r\ge 1$, an \emph{$r$-edge-colouring} of $G$, or simply an \emph{$r$-colouring}, is a function $c:E(G)\to\{1,\dots,r\}$. We think of $\{1,\dots,r\}$ as a set of colours, and occasionally we use the terms \emph{red}, \emph{blue} and \emph{green} if we are focussing on a small number of colours. We may use the terms \emph{edge-colouring} or \emph{colouring} if we do not wish to emphasize the number of colours. For any other undefined terms, the reader is referred to the books \cite{Bol78,Bol98, Die18}.

Let $k$ be a positive integer. A set of internally vertex-disjoint paths connecting two vertices in a graph will simply be called \emph{disjoint}. By Menger's theorem \cite{Men27}, a graph is $k$-connected if and only if every two vertices have $k$ disjoint paths connecting them. An edge-coloured path is \emph{rainbow} if all of its edges have distinct colours. Let $G$ be a $k$-connected graph. An edge-colouring of $G$, not necessarily proper, is \emph{rainbow $k$-connected} if every two vertices have $k$ disjoint rainbow paths connecting them. The \emph{rainbow $k$-connection number} of $G$, denoted by $rc_k(G)$, is the minimum possible number of colours in a rainbow $k$-connected colouring of $G$. Note that $rc_k(G)$ is well-defined if and only if $G$ is $k$-connected. We write $rc(G)$ for $rc_1(G)$. The parameter $rc_k(G)$ was introduced by Chartrand, Johns, McKeon and Zhang (\cite{CJMZ08} for $k = 1$ (2008), and \cite{CJMZ09} for general $k$ (2009)), and has since attracted significant interest from many researchers. For an informative survey and book on the subject of rainbow connection of graphs, see \cite{LSS13,LS12}.

We consider the following two problems.

\begin{pro}\label{prob1}
For integers $r,k\ge 1$, let $t_k(n, r)$ denote the minimum number of edges in a $k$-connected graph $G$ on $n$ vertices with $rc_k(G)\le r$. Determine $t_k(n,r)$.
\end{pro}

\begin{pro}\label{prob2}
For integers $r,k\ge 1$, let $s_k(n, r)$ denote the maximum number of edges in a $k$-connected graph $G$ on $n$ vertices with $rc_k(G)\ge r$. Determine $s_k(n,r)$.
\end{pro}

We note that Problems \ref{prob1} and \ref{prob2} have the following equivalent formulations.

\begin{pro1a}\label{prob1a}
Given integers $r,k\ge 1$, compute the maximum integer $g_k(n, r)$ such that, if $G$ is a $k$-connected graph on $n$ vertices and $|E(G)| \le g_k(n, r)$, then $rc_k(G) \ge r$.
\end{pro1a}

\begin{pro2a}\label{prob2a}
Given integers $r,k\ge 1$, compute the minimum integer $f_k(n, r)$ such that, if $G$ is a $k$-connected graph on $n$ vertices and $|E(G)| \ge f_k(n, r)$, then $rc_k(G) \le r$.
\end{pro2a}
It is easy to see that $t_k(n, r) = g_k(n, r+1) + 1$ and $s_k(n, r) = f_k(n, r - 1) - 1$, whenever the terms are defined. Problems \ref{prob1} and \ref{prob2} are  ``Erd\H{o}s-Gallai type problems''. Loosely speaking, an  Erd\H{o}s-Gallai type problem involves the study of the bounds of a graph parameter (such as the number of edges), among all graphs that satisfy some property. For the case $k=1$, note that if $G$ is a connected graph on $n$ vertices, then we have $|E(G)|\ge n-1$, with equality if and only if $G$ is a tree. Also, we have $1\le \diam(G)\le rc(G)\le n-1$, with $rc(G)= n-1$ if and only if $G$ is a tree. It is easy to see that
\begin{align}
{n\choose 2} &= t_1(n,1)\ge t_1(n,2)\ge\cdots\ge t_1(n,n-1)=t_1(n,n)=\cdots=n-1,\label{t1chain}\\
{n\choose 2} &= s_1(n,1)\ge s_1(n,2)\ge\cdots\ge s_1(n,n-1)=n-1,\label{s1chain}
\end{align}
and thus it suffices to consider $1\le r\le n-1$ in the above problems. Problem \ref{prob1} for $k=1$ was first considered by Schiermeyer \cite{Sch13}, when he determined the values of $t_1(n,r)$ for $\frac{n}{2}\le r\le n-1$, and the asymptotic answer for $r=2$. Subsequently, Bode and Harborth \cite{BH13}, and Li et al.~\cite{LLSZ14} made some improvements, and Lo \cite{Lo14} determined the values of $t_1(n,r)$ for $3\le r<\frac{n}{2}$. Thus Problem \ref{prob1} is essentially completely solved for $k=1$.

\begin{thm}\label{t1thm}\textup{\cite{BH13,Lo14,Sch13}}
Let $1\le r\le n-1$ and $n\ge 3$. Then
\begin{equation}\label{t1eq}
t_1(n,r)=
\left\{
\begin{array}{ll}
\displaystyle {n\choose 2} & \textup{\emph{if} }r=1\textup{\emph{,}}\\[2ex]
(1+o(1))n \log_2 n & \textup{\emph{if} }r=2\textup{\emph{,}}\\[1ex]
\displaystyle \Big\lceil\frac{r(n-2)}{r-1}\Big\rceil & \textup{\emph{if} }r\ge 3.
\end{array}
\right.
\end{equation}
\end{thm}

For $r\ge 2$, the constructions of a graph $G$ on $n$ vertices yielding the upper bounds in (\ref{t1eq}) are as follows. For $r=2$, we take $G$ to be the complete bipartite graph $K_{a,n-a}$ such that $2^{a-1}+a\le n\le 2^a+a$, so that $|E(G)|=(1+o(1))n\log_2n$. For $3\le r\le n-1$, we take $\lfloor\frac{n-3}{r-1}\rfloor$ cycles on $r$ vertices having one vertex in common, and connect the remaining $n - 1 -\lfloor\frac{n-3}{r-1}\rfloor(r-1)\le r$ vertices to the common vertex of the cycles. Then $|E(G)|=\lceil\frac{r(n-2)}{r-1}\rceil $. For these graphs, it was shown in \cite{CJMZ08} and \cite{BH13} that $rc(G)=r$.

Problem \ref{prob2}$'$ for $k=1$ was first considered by Kemnitz and Schiermeyer \cite{KS11}, and their results were subsequently improved by Kemnitz et al.~\cite{KPSW13}, and Li et al.~\cite{LLS13}. When stated in terms of $s_1(n,r)$, we have the following result.

\begin{thm}\label{f1thm}\textup{\cite{KPSW13,KS11,LLS13}}
Let $2\le r\le n-1$. Then $s_1(n,r)\ge {n-r+2\choose 2}+r-3$. Equality holds for $r\in\{2, 3, 4,5, n - 5, n -4, n -3, n -2, n -1\}$.
\end{thm}


The graph $G$ which yields the lower bound in Theorem \ref{f1thm}, as given in \cite{KS11}, is as follows. Take a complete graph $K_{n-r+2}$ and delete one edge $xy$. Then take a path on $r-1$ vertices and identify one end-vertex with $x$. Let $G$ be the resulting graph. We have $|E(G)|= {n-r+2\choose 2}+r-3$, and $rc(G)\ge\diam(G)=r$.

Finally, note that for $1\le k<\ell<n$, we have
\begin{equation}
t_k(n,r)\le t_\ell(n,r),\label{tktleq}
\end{equation}
whenever the terms are defined.

We remark that analogous Erd\H{o}s-Gallai type problems have also been considered for other parameters similar to the rainbow connection number, such as the monochromatic connection number \cite{CLW17}, and rainbow disconnection number \cite{BCLH19,BL19,CDHHZ18}, among others.

In this paper, we will focus on Problems \ref{prob1} and \ref{prob2} for the case $k=2$. This paper is organised as follows. In Section \ref{toolssect}, we gather some auxiliary results about $k$-connected graphs and the rainbow $k$-connection number. In Section \ref{t2(n,r)sect},
we study the function $t_2(n,r)$. Our main result will be Theorem \ref{t2(n,r)thm}, from which we can conclude that $t_2(n,2)=(1+o(1))n\log_2 n$, and $t_2(n,r)$ is linear in $n$ for $r\ge 3$. In Section \ref{s2(n,r)sect}, we provide some remarks about the function $s_2(n,r)$.


\section{Tools and some related results}\label{toolssect}

In this section, we gather some results which will be useful for Sections \ref{t2(n,r)sect} and \ref{s2(n,r)sect}. Throughout this section, let $n>k\ge 2$.

We first note that any $k$-connected graph $G$ on $n$ vertices has at least $\lceil\frac{kn}{2}\rceil$ edges, since $G$ has minimum degree at least $k$. Harary \cite{Har62} gave examples of $k$-connected graphs which show that the value $\lceil\frac{kn}{2}\rceil$ is best possible for the minimum number of edges.

\begin{thm}\label{Harthm} \emph{\cite{Har62}}
Let $n>k\ge 2$, and $G$ be a $k$-connected graph on $n$ vertices. Then $|E(G)|\ge\lceil\frac{kn}{2}\rceil$. Moreover, there exists a $k$-connected graph $H_{n,k}$ on $n$ vertices with $|E(H_{n,k})|=\lceil\frac{kn}{2}\rceil$.
\end{thm}

Since\, we\, will\, mainly\, be\, studying\, $2$-connected\, graphs,\, we\, state\, the\, well-known\, \emph{ear decomposition theorem}, which gives a characterisation for the structure of $2$-connected graphs. See for example \cite{Die18} (Ch.~3, Proposition 3.1.1).

\begin{thm}[Ear\, decomposition]\label{EDthm}
A\, graph\, $G$\, is\, $2$-connected\, if\, and\, only\, if\, $G$\, can\, be constructed from a cycle by successively adding ears to graphs that are already constructed.
\end{thm}

\begin{rmk} \textup{It is easy to see from Theorem \ref{EDthm} that, if $G$ is a $2$-connected graph on $n$ vertices and $q\ge 0$, then $G$ is constructed by successively adding $q$ ears, starting with an initial cycle, if and only if $|E(G)|=n+q$.}
\end{rmk}

We also have the following result from the book of Bollob\'as \cite{Bol78} (Ch.~IV, Theorem 2.8).

\begin{thm}\label{BBthm}\textup{\cite{Bol78}}
Let $m\ge 2$, and $G$ be a $2$-connected graph on $n$ vertices with diameter at most $2m$. Then
\[
|E(G)|\ge \Big(1+\frac{1}{2m-1}\Big)n-4(4m-2)^{m-1}.
\]
\end{thm}

Now, we consider some results about the rainbow $k$-connection number. Note that for any $k$-connected graph $G$, we have $rc_k(G)\ge 2$. When $n$ is sufficiently large, Chartrand et al.~\cite{CJMZ09} showed that equality holds for the complete graph $K_n$.

\begin{thm}\label{CJMZthm}\emph{\cite{CJMZ09}}
\indent\\[-3ex]
\begin{enumerate}
\item[(a)] $rc_2(K_n)=2$ for $n\ge 4$.
\item[(b)] $rc_k(K_n)=2$ for $k\ge 2$ and $n\ge (k+1)^2$.
\end{enumerate}
\end{thm}

To see (a), we have a rainbow $2$-connected $2$-colouring of $K_n$ as follows. For $n\ge 5$, we colour a Hamilton cycle of $K_n$ with red, and all remaining edges with blue. For $K_4$, we take the $2$-colouring where the colours red and blue both induce a path of length three.

Next, we have the following result of Li and Liu \cite{LL13} about the rainbow $2$-connection number of $2$-connected graphs. This result is analogous to the fact that $rc(G)\le n-1$ if $G$ is a connected graph on $n$ vertices.

\begin{thm}\label{LLthm}\textup{\cite{LL13}}
Let $G$ be a $2$-connected graph on $n\ge 3$ vertices. Then $rc_2(G) \le n$, with equality if and only if $G$ is the cycle on $n$ vertices.
\end{thm}

Finally, we prove the following useful lemma.

\begin{lem}\label{lm1}
Let $G$ be a $2$-connected graph, and $P$ be a path in $G$ of length at least three such that, all internal vertices of $P$ have degree two in $G$. Then for any rainbow $2$-connected colouring of $G$, the edges of $P$ must be rainbow coloured.
\end{lem}

\begin{proof}
Suppose that there exists a rainbow $2$-connected colouring of $G$, and a path $P$ in $G$ as described in the lemma which is not rainbow coloured. We have $P=x_0x_1\cdots x_\ell$ for some $\ell\ge 3$, where $\deg_G(x_i)=2$ for all $1\le i\le\ell-1$. There exist two edges $x_{i-1}x_i$ and $x_{j-1}x_j$ with the same colour, for some $1\le i<j\le\ell$. If $i<j-1$, let $u=x_i$ and $v=x_{j-1}$. If $i=j-1$, let $u=x_{i-1}$ and $v=x_j$. Note that since $\ell\ge 3$, at least one of $u$ and $v$ is an internal vertex of $P$. We see that there do not exist two disjoint rainbow paths connecting $u$ and $v$, a contradiction.
\end{proof}


\section{The function $t_2(n,r)$}\label{t2(n,r)sect}

Let $n>k\ge 2$, and $G$ be a $k$-connected graph on $n$ vertices. By Theorem \ref{CJMZthm}(b), we see that for sufficiently large $n$ ($n\ge (k+1)^2$ will do), $t_k(n,2)$ is well-defined. Together with Theorem \ref{Harthm},  we have the following analogue to (\ref{t1chain})
\[
{n\choose 2} \ge t_k(n,2)\ge t_k(n,3)\ge\cdots \ge t_k\Big(n,\Big\lceil\frac{kn}{2}\Big\rceil\Big) = t_k\Big(n,\Big\lceil\frac{kn}{2}\Big\rceil+1\Big)= \cdots=\Big\lceil\frac{kn}{2}\Big\rceil.
\]

In this section, we study the function $t_2(n,r)$. For $n\ge 3$, the cycle $C_n$ is the unique $2$-connected graph on $n$ vertices with the minimum number of edges, which is $n$. Note that $rc_2(C_n)=n$. Thus for $n\ge 4$, Theorem \ref{CJMZthm}(a) implies
\begin{equation}\label{t2chain}
{n\choose 2} \ge t_2(n,2)\ge t_2(n,3)\ge\cdots \ge t_2(n,n)=t_2(n,n+1)=\cdots=n.
\end{equation}
It suffices to consider $t_2(n,r)$ for $2\le r\le n-1$. The following theorem is our main result.

\begin{thm}\label{t2(n,r)thm}
\indent\\[-3ex]
\begin{enumerate}
\item[(a)] $t_2(n,2)=(1+o(1))n\log_2 n$.
\item[(b)]
\[
t_2(n,r)\ge
\left\{
\begin{array}{ll}
\displaystyle\frac{4}{3}n-24 & \textup{\emph{if} }r\in\{3,4\}\textup{\emph{ and }}n\ge 18\textup{\emph{,}}\\[2ex]
\displaystyle\frac{6}{5}n-\frac{1}{5}r(r-1) & \textup{\emph{if} }5\le r\le n-1\textup{\emph{ and }}\displaystyle n>\frac{1}{6}r(r-1).
\end{array}
\right.
\]
\item[(c)]
\[
t_2(n,r)\le
\left\{
\begin{array}{ll}
\displaystyle\frac{5}{2}n-5 & \textup{\emph{if} }r\in\{3,4\}\textup{\emph{ and }}n\ge r+1\textup{\emph{,}}\\[2ex]
\displaystyle\frac{7n-19}{3} & \textup{\emph{if} }r=5\textup{\emph{ and }}n\ge 7\textup{\emph{,}}\\[2ex]
  2n-r+2 &\textup{\emph{if} }6\le r\le n-3.
\end{array}
\right.
\]
\item[(d)] $t_2(n,n-2)=n+2$ for $n\ge 6$, and $t_2(n,n-1)=n+1$ for $n\ge 4$.
\end{enumerate}
\end{thm}

\begin{proof}
(a) We have $t_2(n,2)\ge t_1(n,2)\ge n(\log_2 n -4\log_2 \log_2 n - 2)$ for sufficiently large $n$, where the first inequality follows from (\ref{tktleq}), and the second inequality was proved by Li et al.~\cite{LLSZ14}. Now we construct a graph $G$ on $n$ vertices, with a $2$-colouring $c$, as follows. Take the complete bipartite graph with classes $A$ and $B$, where $|A|=a\ge 4$ and $|B|=n-a$, such that $2^{a-2}+a-1\le n\le 2^{a-1}+a-1$. Add a complete graph on $A$. It is easy to obtain $|E(G)|=(1+o(1))n\log_2 n$. Now, assign to the vertices of $B$ distinct $(1,2)$-vectors of length $a$ with a positive even number of $2$s, such that the $a-1$ vectors $(2,2,1,\dots,1),(2,1,2,1,\dots,1),\dots,(2,1,\dots,1,2)$ are all present. Note that the assignment is possible, since $|B|=n-a\ge 2^{a-2}-1\ge a-1$, so that the $a-1$ aforementioned vectors can all be assigned; and $|B|=n-a\le 2^{a-1}-1$, whence there are a total of $2^{a-1}-1$ assignable vectors. Let $A=\{u_1,\dots,u_a\}$. For $u_i\in A$ and $v\in B$, let $c(u_iv)={\vec{v}}_i$, where $\vec{v}$ denotes the vector assigned to $v$, and $\vec{v}_i$ denotes the $i$th component of $\vec{v}$. Let $c(u_iu_j)=1$ for all $i,j$. Then $c$ is a rainbow $2$-connected colouring of $G$. Indeed, for two vertices of $G$, we may find two disjoint rainbow paths connecting them as follows.
\begin{itemize}
\item If $u_i,u_j\in A$, then since the $a-1$ aforementioned vectors are present, we may choose $v\in B$ such that ${\vec{v}}_i\neq {\vec{v}}_j$, and take the paths $u_iu_j$ and $u_ivu_j$.
\item If $v,w\in B$, then there exist $i,j$ such that ${\vec{v}}_i\neq {\vec{w}}_i$ and ${\vec{v}}_j\neq {\vec{w}}_j$, and we may take the paths $vu_iw$ and $vu_jw$.
\item  If $u_i\in A$ and $v\in B$, then we may choose $u_j\in A$ such that ${\vec{v}}_j=2$, and take the paths $u_iv$ and $u_iu_jv$.
\end{itemize}
Thus $rc_2(G)=2$, and $t_2(n,2)\le (1+o(1))n\log_2 n$.\\[1ex]
\indent(b) Let $r\ge 3$, and $G$ be a $2$-connected graph on $n$ vertices with $rc_2(G)\le r$. Note that since $\diam(G)\le rc_2(G)\le 2\lceil\frac{r}{2}\rceil$, setting $m=\lceil\frac{r}{2}\rceil$ in Theorem \ref{BBthm} gives $|E(G)|\ge \big(1+\frac{1}{2\lceil r/2\rceil-1}\big)n-c_r$ for some constant $c_r>0$. This implies 
\begin{equation}
t_2(n,r)\ge \Big(1+\frac{1}{2\lceil r/2\rceil-1}\Big)n-c_r.\label{t2diameq}
\end{equation}
For $r\ge 5$, we will see that the constant $1+\frac{1}{2\lceil r/2\rceil-1}$ in (\ref{t2diameq}) can be improved to $\frac{6}{5}$. For $r\in\{3,4\}$ and $n\ge 18$, setting $m=2$ in Theorem \ref{BBthm} gives $|E(G)|\ge \frac{4}{3}n-24$. Therefore, $t_2(n,r)\ge \frac{4}{3}n-24$.


Now, suppose that $5\le r\le n-1$ and $n>\frac{1}{6}r(r-1)$. Note that $rc_2(G)\le n-1$, so that $G\neq C_n$. We prove that $|E(G)|\ge \frac{6}{5}n-\frac{1}{5}r(r-1)$, which implies $t_2(n,r)\ge \frac{6}{5}n-\frac{1}{5}r(r-1)$. Let $V_2\subset V(G)$ be the subset defined by $V_2=\{v\in V(G):\deg_G(v)=2\}$. Note that $G[V_2]$ is a linear forest, i.e., $G[V_2]$ is a union of vertex-disjoint paths.
For an integer $i\ge 1$, let $\mathcal{P}_i$ be the family of path components in $G[V_2]$ such that $|V(P)|=i$ for every $P\in \mathcal{P}_i$. We have $V_2=\bigcup_{i\ge 1}\bigcup_{P\in \mathcal{P}_i}V(P)$.
By the definition of $\mathcal{P}_i$, together with the assumption that $G$ is $2$-connected, the following two properties hold:
\begin{enumerate}
\item[(i)] For any $P\in \mathcal{P}_i$ and $Q\in \mathcal{P}_j$ (possibly $i=j$), we have $E_G(V(P), V(Q))=\emptyset$.
\item[(ii)] For any $P=x_1\cdots x_i\in \mathcal{P}_i$, there exist two distinct vertices $x_0,x_{i+1}\in V(G)\setminus V_2$ such that $E_G(V(P), V(G)\setminus V(P))=\{x_1x_0, x_ix_{i+1}\}$, with $\deg_G(x_0)\ge 3$ and $\deg_G(x_{i+1})\ge 3$.
\end{enumerate}

Let $P=x_1\cdots x_i\in \mathcal{P}_i$ be a path, for some $i$. In view of (ii), we will often look at the extended path of $P$ containing $x_0$ and $x_{i+1}$. Thus, let $P^+$ be the path $P^+=x_0x_1\cdots x_ix_{i+1}$. Note that the paths of $\big\{P^+:P\in\bigcup_{i\ge 1}\mathcal{P}_i\big\}$ are pairwise edge-disjoint. Now, since $rc_2(G)\le r$, we may fix a rainbow $2$-connected $r$-colouring on $G$. By Lemma \ref{lm1}, we have $P^+$ is a rainbow path for any $P\in\bigcup_{i\ge 2}\mathcal{P}_i$. This implies that $\mathcal{P}_i=\emptyset$ for $i\ge r$.

\begin{clm}\label{c1}
We have
\[
\bigg|V_2\setminus \bigcup_{P\in \mathcal{P}_1}V(P)\bigg|=\sum_{i=2}^{r-1}i|\mathcal{P}_i|\le \binom{r}{2}.
\]
\end{clm}

\begin{proof}
First, suppose that there exist two paths $P\in\mathcal P_i$ and $Q\in\mathcal P_j$, for some $2\le i,j\le r-1$, such that $P^+$ and $Q^+$ both contain two consecutive edges using the same pair of colours, say red and blue. Let $u\in V(P)$ and $v\in V(Q)$ be the two vertices between the red and blue edges, and note that $\deg_G(u)=\deg_G(v)=2$. Since $i\ge 2$, we may choose a neighbour $w$ of $u$ such that $\deg_G(w)=2$. Then $w$ and $v$ are not connected by two disjoint rainbow paths, a contradiction.

Now,\, $\bigcup_{i=2}^{r-1}\bigcup_{P\in \mathcal{P}_i}P^+$\, contains\, $\sum_{i=2}^{r-1}i|\mathcal P_i|$\, pairs\, of\, edges,\, with\, each\, pair\, being\, two consecutive edges of some path $P^+$. These pairs of edges use distinct pairs of colours, therefore
\[
\bigg|V_2\setminus \bigcup_{P\in \mathcal{P}_1}V(P)\bigg|=\sum_{i=2}^{r-1}i|\mathcal{P}_i|\le \binom{r}{2}.
\]

This proves Claim \ref{c1}.
\end{proof}

If $|\mathcal{P}_1|\ge \frac{3}{5}n-\frac{1}{10}r(r-1)>0$, then the bipartite subgraph of $G$ induced by  $V(G)\setminus V_2$ and $\{V(P):P\in\mathcal P_1\}$ has $2|\mathcal P_1|\ge\frac{6}{5}n-\frac{1}{5}r(r-1)$ edges. Otherwise, if $|\mathcal{P}_1|\le \frac{3}{5}n-\frac{1}{10}r(r-1)$, then by Claim~\ref{c1},
\begin{align*}
|E(G)|&\ge \frac{1}{2}(2|V_2|+3(n-|V_2|))=\frac{3}{2}n-\frac{1}{2}|V_2|=\frac{3}{2}n-\frac{1}{2}\sum_{i=2}^ri|\mathcal{P}_i|-\frac{1}{2}|\mathcal{P}_1|\\
&\ge \frac{3}{2}n-\frac{1}{2}\binom{r}{2}-\frac{3}{10}n+\frac{1}{20}r(r-1)=\frac{6}{5}n-\frac{1}{5}r(r-1).
\end{align*}
\indent(c) We first consider the case $6\le r\le n-3$. We construct a graph $G_{n,r}$ on $n$ vertices with an $r$-colouring, as follows. First, let $F_6$ be the graph on six vertices $x,y,a,b,c,d$, with the $6$-colouring as shown in Figure 1(a). Then, let $H_m$ be the graph obtained by taking $m\ge 1$ identically $6$-coloured copies of $F_6$, say $F_6^1,\dots, F_6^m$, and identifying the vertices $x$ and $y$. Let $a_i,b_i,c_i,d_i$ be the vertices of $F_6^i$ corresponding to $a,b,c,d$, for $1\le i\le m$. See Figure 1(b) for the case of $H_2$. Next, let $H_{m,r}$ be obtained from $H_m$ by adding an $(r-4)$-ear $P$ at $x$ and $y$. Colour the edges of $P$ with colours $5,6,\dots, r$ such that colours $5$ and $6$ are incident with $x$ and $y$. See Figure 1(c).\\[1ex]
\[ \unit = 0.7cm
\varline{450}{0.6}
\dl{-5.8}{3}{-6.8}{1}\dl{-6.8}{1}{-7.8}{3}\dl{-5.8}{3}{-6.8}{5}\dl{-6.8}{5}{-7.8}{3}
\dl{-5.8}{3}{-6.8}{2.5}\dl{-7.8}{3}{-6.8}{2.5}\dl{-5.8}{3}{-6.8}{3.5}\dl{-7.8}{3}{-6.8}{3.5}
\pt{-5.8}{3}\pt{-6.8}{1}\pt{-7.8}{3}\pt{-6.8}{5}\pt{-6.8}{2.5}\pt{-6.8}{3.5}
\ptll{-7.9}{3}{a}\ptlr{-5.7}{3}{c}\ptlu{-6.8}{3.5}{b}\ptld{-6.8}{2.5}{d}\ptlu{-6.8}{5}{x}\ptld{-6.8}{1}{y}
\ptll{-7.3}{1.9}{2}\ptlr{-6.3}{1.9}{1}\ptll{-7.3}{4.1}{1}\ptlr{-6.3}{4.1}{2}\ptlu{-7.2}{3.2}{3}\ptlu{-6.4}{3.2}{4}\ptld{-7.2}{2.75}{5}\ptld{-6.4}{2.75}{6}
\ptlu{-8.6}{5}{\textup{\normalsize (a)}}
\ptlu{-6.8}{-0.4}{\textup{\normalsize $F_6$}}
%
\dl{-0.3}{1}{-2.8}{3}\dl{-0.3}{1}{2.2}{3}\dl{-0.3}{5}{-2.8}{3}\dl{-0.3}{5}{2.2}{3}\dl{-0.3}{1}{-1}{3}\dl{-0.3}{1}{0.4}{3}\dl{-0.3}{5}{-1}{3}\dl{-0.3}{5}{0.4}{3}
\dl{-1}{3}{-1.5}{2.5}\dl{-1.5}{2.5}{-2.8}{3}\dl{-1}{3}{-1.5}{3.5}\dl{-1.5}{3.5}{-2.8}{3}
\dl{0.4}{3}{0.9}{2.5}\dl{0.9}{2.5}{2.2}{3}\dl{0.4}{3}{0.9}{3.5}\dl{0.9}{3.5}{2.2}{3}
\pt{-2.8}{3}\pt{-0.3}{1}\pt{2.2}{3}\pt{-0.3}{5}\pt{0.4}{3}\pt{0.9}{2.5}\pt{0.9}{3.5}\pt{-1}{3}\pt{-1.5}{2.5}\pt{-1.5}{3.5}
\ptll{-2.8}{3}{a_1}\ptlr{-0.95}{3}{c_1}\ptlu{-1.3}{3.45}{b_1}\ptld{-1.3}{2.55}{d_1}\ptlu{-0.3}{5}{x}\ptld{-0.3}{1}{y}
\ptll{-1.6}{1.9}{2}\ptll{-1.6}{4.1}{1}\ptll{-0.65}{1.9}{1}\ptll{-0.65}{4.1}{2}
\ptlu{-2}{3}{3}\ptlu{-1.2}{3}{4}\ptld{-2}{3}{5}\ptld{-1.2}{3}{6}
\ptlr{1}{1.9}{1}\ptlr{1}{4.1}{2}
\ptlr{2.25}{3}{c_2}\ptll{0.35}{3}{a_2}\ptlu{0.7}{3.45}{b_2}\ptld{0.7}{2.55}{d_2}
\ptlr{0}{1.9}{2}\ptlr{0}{4.1}{1}
\ptlu{0.6}{3}{3}\ptlu{1.4}{3}{4}\ptld{1.4}{3}{6}\ptld{0.6}{3}{5}
\ptlr{1}{1.9}{1}\ptlr{1}{4.1}{2}
\ptlu{-3.6}{5}{\textup{\normalsize (b)}}
\ptlu{-0.3}{-0.4}{\textup{\normalsize $H_2$}}
%
\dotline
\ellipse{6.7}{3}{1.5}{2}
\varline{450}{0.6}
\dl{6.7}{1}{8.2}{1}\dl{8.2}{1}{8.9}{1.7}\dl{6.7}{5}{8.2}{5}\dl{8.2}{5}{8.9}{4.3}\dl{8.9}{1.7}{8.9}{2.2}\dl{8.9}{4.3}{8.9}{3.8}
\ptlu{8.9}{2.58}{\vdots}
\pt{6.7}{1}\pt{6.7}{5}\pt{8.2}{1}\pt{8.2}{5}\pt{8.9}{1.7}\pt{8.9}{4.3}
\ptlu{6.7}{2.65}{\textup{\normalsize $H_m$}}\ptlu{6.7}{5}{x}\ptld{6.7}{1}{y}\ptlu{7.45}{5}{5}\ptld{7.45}{1}{6}
\ptlu{8.7}{4.55}{7}\ptld{8.7}{1.45}{r}
\ptlr{8.95}{4}{\textup{\normalsize $P$}}
\ptlu{4.4}{5}{\textup{\normalsize (c)}}
\ptlu{7.05}{-0.4}{\textup{\normalsize $H_{m,r}$}}
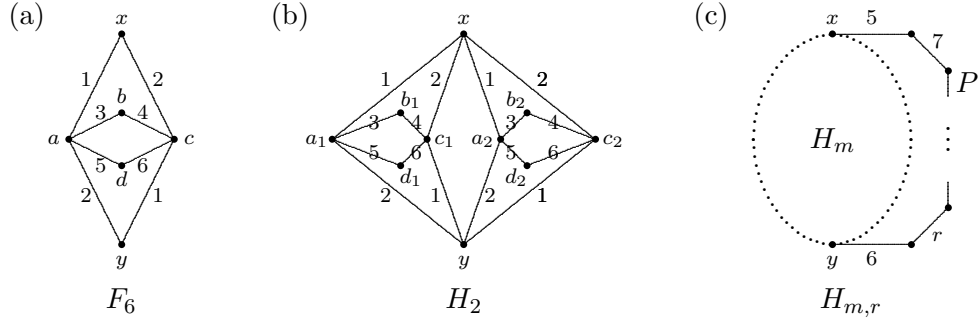
\ptlu{0}{-1.6}{\textup{\normalsize Figure 1. The graphs $F_6$, $H_2$ and $H_{m,r}$}}
\]\\[1ex]
\indent We have $|V(H_{m,r})|=4m+r-3$. Let $m$ and $b$ be integers such that  $n=4m+r-3+b$, where  $0\le b\le 3$. Note that $m\ge 1$ since $n\ge r+3$. Let $G_{n,r}$ be the graph obtained from $H_{m,r}$ by adding $2$-ears $Q_1,\dots, Q_b$ at $x$ and $y$, with the colouring as shown in Figure 2. Note that $G_{n,r}=H_{m,r}$ if $b=0$. If $b\ge 1$, let $w_j$ be the middle vertex of $Q_j$ for $1\le j\le b$. We have $G_{n,r}$ is a $2$-connected graph on $n$ vertices, and $|E(G_{n,r})|=8m+r-4+2b=2n-r+2$.\\[1ex]
\[ \unit = 0.7cm
\varline{450}{0.6}
\dotline
\ellipse{-6.3}{2.5}{1}{1.5}
\varline{450}{0.6}
\bez{-6.3}{1}{-7.3}{1}{-8.3}{2.5}\bez{-6.3}{4}{-7.3}{4}{-8.3}{2.5}
\pt{-6.3}{1}\pt{-8.3}{2.5}\pt{-6.3}{4}
\ptlu{-6.3}{4}{x}\ptld{-6.3}{1}{y}\ptll{-8.3}{2.5}{w_1}
\ptll{-7.5}{1.4}{1}\ptll{-7.8}{3.2}{1}
\ptlu{-6.3}{2.15}{\textup{\normalsize $H_{m,r}$}}
\ptlu{-7.5}{3.7}{\textup{\normalsize $Q_1$}}
\ptlu{-6.8}{-0.4}{\textup{\normalsize $b=1$}}
%
\dotline
\ellipse{-1.3}{2.5}{1}{1.5}
\varline{450}{0.6}
\bez{-1.3}{1}{-2.3}{1}{-3.3}{2.5}\bez{-1.3}{4}{-2.3}{4}{-3.3}{2.5}\bez{-1.3}{1}{-0.3}{1}{0.7}{2.5}\bez{-1.3}{4}{-0.3}{4}{0.7}{2.5}
\pt{0.7}{2.5}\pt{-3.3}{2.5}\pt{-1.3}{1}\pt{-1.3}{4}
\ptlu{-1.3}{4}{x}\ptld{-1.3}{1}{y}\ptll{-3.3}{2.5}{w_1}\ptlr{0.7}{2.5}{w_2}
\ptll{-2.5}{1.4}{1}\ptll{-2.8}{3.2}{1}\ptlr{-0.1}{1.4}{2}\ptlr{0.2}{3.2}{2}
\ptlu{-1.3}{2.15}{\textup{\normalsize $H_{m,r}$}}
\ptlu{-2.5}{3.7}{\textup{\normalsize $Q_1$}}\ptlu{-0.1}{3.7}{\textup{\normalsize $Q_2$}}
\ptlu{-1.3}{-0.4}{\textup{\normalsize $b=2$}}
%
\dotline
\ellipse{6.2}{2.5}{1}{1.5}
\varline{450}{0.6}
\bez{6.2}{1}{5.2}{1}{4.2}{2.5}\bez{6.2}{4}{5.2}{4}{4.2}{2.5}\bez{6.2}{1}{7.2}{1}{8.2}{2.5}\bez{6.2}{4}{7.2}{4}{8.2}{2.5}\bez{6.2}{4}{3.4}{4}{3.2}{2.5}\bez{6.2}{1}{3.4}{1}{3.2}{2.5}
\pt{3.2}{2.5}\pt{4.2}{2.5}\pt{8.2}{2.5}\pt{6.2}{1}\pt{6.2}{4}
\ptlu{6.2}{4}{x}\ptld{6.2}{1}{y}\ptll{4.2}{2.5}{w_1}\ptlr{8.2}{2.5}{w_2}\ptll{3.2}{2.5}{w_3}
\ptlu{5}{1.5}{1}\ptld{5}{3.45}{1}\ptlr{7.4}{1.4}{2}\ptlr{7.7}{3.2}{2}\ptll{3.75}{1.4}{3}\ptll{3.4}{3.2}{4}
\ptlu{6.2}{2.15}{\textup{\normalsize $H_{m,r}$}}
\ptlu{4.15}{2.8}{\textup{\normalsize $Q_1$}}\ptlu{7.4}{3.7}{\textup{\normalsize $Q_2$}}\ptlu{4}{3.7}{\textup{\normalsize $Q_3$}}
\ptlu{5.7}{-0.4}{\textup{\normalsize $b=3$}}
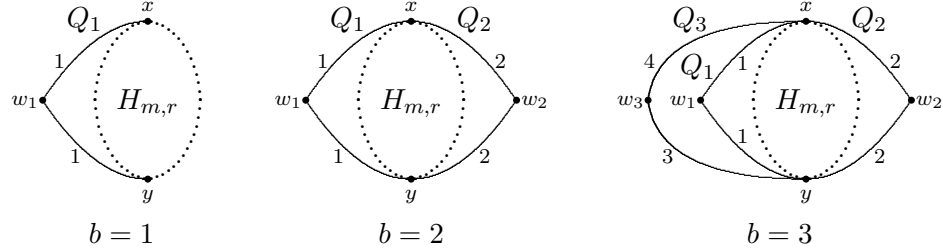
\ptlu{0}{-1.6}{\textup{\normalsize Figure 2. The graph $G_{n,r}$ for $b\in\{1,2,3\}$}}
\]\\[1ex]
\indent To complete the third part of (c), it remains to prove the following claim.
\begin{clm}\label{c2}
The $r$-colouring defined on $G_{n,r}$ is rainbow $2$-connected.
\end{clm}
\begin{proof}
Let $u,v\in V(G_{n,r})$. We show that there are two disjoint rainbow paths connecting $u$ and $v$. This is easy to check if at least one of $u,v$ is $x$ or $y$; or if $u,v$ belong to the same copy of $F_6$; or $u,v\in V(P)$; or $u,v\in\{w_1,w_2,w_3\}$; or $u\in\{w_1,w_2,w_3\}$ and $v\in V(P)$. It remains to check the following cases.
\begin{itemize}
\item Let $u$ and $v$ belong to two different copies of $F_6$. It suffices to consider the cases $(u,v)=(a_1,a_2), (a_1,b_2), (a_1,c_2), (b_1,b_2), (b_1,d_2)$. By consulting Figure 1(b), it is easy to check that for these cases, $u$ and $v$ are connected by two disjoint rainbow paths.
\item Let $u$ belong to a copy of $F_6$, say $F_6^1$, and $v\in V(P)\setminus\{x,y\}$. It is easy to check that $u$ and $v$ are connected by two disjoint rainbow paths. In the case $u=d_1$, we take the paths $ua_1yPv$ and $uc_1xPv$.
\item Let $u$ belong to a copy of $F_6$, say $F_6^1$, and $v\in\{w_1,w_2,w_3\}$. If $(u,v)=(a_1,w_1)$, we take $ub_1c_1xv$ and $uyv$. The case $(u,v)=(a_1,w_2)$ is similar, and the case $(u,v)=(a_1,w_3)$ is easy. The case $u=c_1$ and $v\in \{w_1,w_2,w_3\}$ is similar to these previous three cases. If $u\in\{b_1,d_1\}$ and $v=w_1$, we take $uc_1xv$ and $ua_1yv$. If $u\in\{b_1,d_1\}$ and $v\in\{w_2,w_3\}$, we take $ua_1xv$ and $uc_1yv$.
\end{itemize}

This proves Claim \ref{c2}.
\end{proof}

Next, suppose that $r=5$. Let $F_5$ be the $5$-coloured graph as shown in Figure 3(a). As before, construct the graph $H_m$ by taking $m\ge 1$ copies of $F_5$ and identifying the vertices $x$ and $y$. Then, let $H_{m,5}$ be the graph obtained from $H_m$ by adding a $3$-ear at $x$ and $y$, with the colouring as shown in Figure 3(b). We have $|V(H_{m,5})|=3m+4$. Let $m$ and $b$ be integers such that $n=3m+4+b$, where $0\le b\le 2$. Note that $m\ge 1$ since $n\ge 7$. Let $G_{n,5}$ be the graph obtained from $H_{m,5}$ by adding $2$-ears $Q_1,\dots, Q_b$ at $x$ and $y$, with both edges of $Q_i$ given colour $i$ for $1\le i\le b$. See Figure 3(c). Note that $G_{n,5}=H_{m,5}$ if $b=0$. We have $G_{n,5}$ is a $2$-connected graph on $n$ vertices, and $|E(G_{n,5})|=7m+3+2b\le\frac{7n-19}{3}$. By a similar argument as in Claim \ref{c2}, the $5$-colouring of $G_{n,5}$ is rainbow $2$-connected. This proves the second part of (c).\\[-1ex]
\[ \unit = 0.7cm
\varline{450}{0.6}
\dl{-6.2}{3}{-7.2}{1}\dl{-7.2}{1}{-8.2}{3}\dl{-6.2}{3}{-7.2}{5}\dl{-7.2}{5}{-8.2}{3}
\dl{-6.2}{3}{-7.2}{3.5}\dl{-8.2}{3}{-7.2}{3.5}\dl{-8.2}{3}{-6.2}{3}
\pt{-6.2}{3}\pt{-7.2}{1}\pt{-8.2}{3}\pt{-7.2}{5}\pt{-7.2}{3.5}
\ptll{-8.3}{3}{a}\ptlr{-6.1}{3}{c}\ptlu{-7.2}{3.5}{b}\ptlu{-7.2}{5}{x}\ptld{-7.2}{1}{y}
\ptll{-7.7}{1.9}{2}\ptlr{-6.7}{1.9}{1}\ptll{-7.7}{4.1}{1}\ptlr{-6.7}{4.1}{2}\ptlu{-7.6}{3.2}{3}\ptlu{-6.8}{3.2}{4}\ptld{-7.2}{3}{5}
\ptlu{-9}{5}{\textup{\normalsize (a)}}
\ptlu{-7.2}{-0.4}{\textup{\normalsize $F_5$}}
%
\dotline
\ellipse{-2.7}{3}{1}{1.5}
\varline{450}{0.6}
\dl{-2.7}{1.5}{-1.2}{1.5}\dl{-2.7}{4.5}{-1.2}{4.5}\dl{-1.2}{1.5}{-1.2}{4.5}
\pt{-2.7}{1.5}\pt{-2.7}{4.5}\pt{-1.2}{1.5}\pt{-1.2}{4.5}
\ptlu{-2.7}{4.5}{x}\ptld{-2.7}{1.5}{y}\ptlu{-1.95}{4.5}{3}\ptld{-1.95}{1.5}{4}\ptlr{-1.15}{3}{5}
\ptlu{-2.7}{2.65}{\textup{\normalsize $H_m$}}
\ptlu{-4.5}{5}{\textup{\normalsize (b)}}
\ptlu{-2.45}{-0.4}{\textup{\normalsize $H_{m,5}$}}
%
\dotline
\ellipse{3.3}{3}{1}{1.5}
\varline{450}{0.6}
\bez{3.3}{1.5}{2.3}{1.5}{1.3}{3}\bez{3.3}{4.5}{2.3}{4.5}{1.3}{3}
\pt{3.3}{1.5}\pt{1.3}{3}\pt{3.3}{4.5}
\ptlu{3.3}{4.5}{x}\ptld{3.3}{1.5}{y}\ptll{1.3}{3}{w_1}
\ptll{2.1}{1.9}{1}\ptll{1.8}{3.7}{1}
\ptlu{3.3}{2.65}{\textup{\normalsize $H_{m,5}$}}
\ptlu{2.1}{4.2}{\textup{\normalsize $Q_1$}}
%
\dotline
\ellipse{7.3}{3}{1}{1.5}
\varline{450}{0.6}
\bez{7.3}{1.5}{6.3}{1.5}{5.3}{3}\bez{7.3}{4.5}{6.3}{4.5}{5.3}{3}\bez{7.3}{1.5}{8.3}{1.5}{9.3}{3}\bez{7.3}{4.5}{8.3}{4.5}{9.3}{3}
\pt{9.3}{3}\pt{5.3}{3}\pt{7.3}{1.5}\pt{7.3}{4.5}
\ptlu{7.3}{4.5}{x}\ptld{7.3}{1.5}{y}\ptll{5.3}{3}{w_1}\ptlr{9.3}{3}{w_2}
\ptll{6.1}{1.9}{1}\ptll{5.8}{3.7}{1}\ptlr{8.5}{1.9}{2}\ptlr{8.8}{3.7}{2}
\ptlu{7.3}{2.65}{\textup{\normalsize $H_{m,5}$}}
\ptlu{6.1}{4.2}{\textup{\normalsize $Q_1$}}\ptlu{8.5}{4.2}{\textup{\normalsize $Q_2$}}
\ptlu{5.3}{-0.4}{\textup{\normalsize $G_{n,5}$ for $b\in\{1,2\}$}}
\ptlu{0.5}{5}{\textup{\normalsize (c)}}
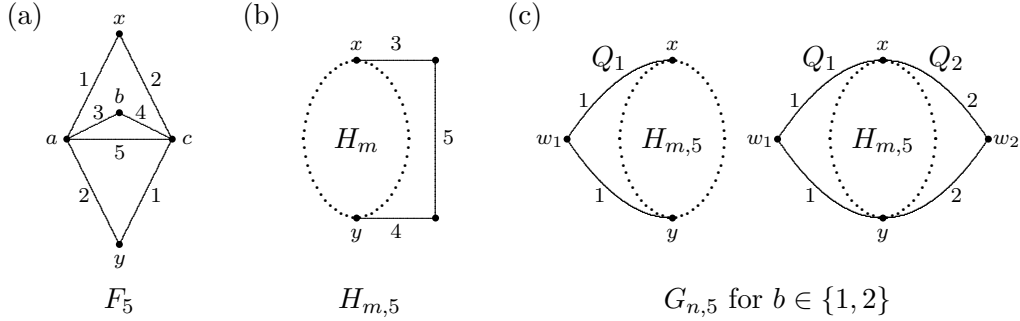
\ptlu{0}{-1.4}{\textup{\normalsize Figure 3. The graphs $F_5$, $H_{m,5}$ and $G_{n,5}$}}
\]\\[-0.5ex]
\indent Finally, let $r\in\{3,4\}$. Let $F_3$ be the $3$-coloured graph as shown in Figure 4(a). Let $H_m$ be the graph obtained by taking $m\ge 1$ copies of $F_3$ and identifying the vertices $x$ and $y$. We have $|V(H_m)|=2m+2$. Let $n=2m+2+b$, where $0\le b\le 1$, and note that $m\ge 1$ since $n\ge r+1\ge 4$. If $n$ is even, set $G_{n,3}=H_m$. If $n$ is odd, let $G_{n,3}$ be obtained by adding a $2$-ear to $H_m$ at $x$ and $y$, with both edges of the ear given colour $3$. See Figure 4(b). We have $G_{n,3}$ is a $2$-connected graph on $n$ vertices, and $|E(G_{n,3})|=5m+2b\le\frac{5}{2}n-5$. It is easy to check that the $3$-colouring of $G_{n,3}$ is rainbow $2$-connected. This proves the first part of (c).\\[-0.5ex]
\[ \unit = 0.7cm
\varline{450}{0.6}
\dl{-1.9}{2.5}{-2.9}{1}\dl{-2.9}{1}{-3.9}{2.5}\dl{-1.9}{2.5}{-2.9}{4}\dl{-2.9}{4}{-3.9}{2.5}\dl{-1.9}{2.5}{-3.9}{2.5}
\pt{-1.9}{2.5}\pt{-2.9}{1}\pt{-3.9}{2.5}\pt{-2.9}{4}
\ptll{-4}{2.5}{a}\ptlr{-1.85}{2.5}{c}\ptlu{-2.9}{4}{x}\ptld{-2.9}{1}{y}
\ptll{-3.4}{1.65}{2}\ptlr{-2.4}{1.65}{1}\ptll{-3.4}{3.35}{1}\ptlr{-2.4}{3.35}{2}\ptld{-2.9}{2.5}{3}
\ptlu{-4.7}{4}{\textup{\normalsize (a)}}
\ptlu{-2.9}{-0.4}{\textup{\normalsize $F_3$}}
%
\dotline
\ellipse{3.9}{2.5}{1}{1.5}
\varline{450}{0.6}
\bez{3.9}{1}{2.9}{1}{1.9}{2.5}\bez{3.9}{4}{2.9}{4}{1.9}{2.5}
\pt{3.9}{1}\pt{3.9}{4}\pt{1.9}{2.5}
\ptlu{3.9}{4}{x}\ptld{3.9}{1}{y}\ptll{1.85}{2.5}{w}
\ptll{2.7}{1.35}{3}\ptll{2.7}{3.6}{3}
\ptlu{1.1}{4}{\textup{\normalsize (b)}}
\ptlu{3.9}{2.15}{\textup{\normalsize $H_m$}}
\ptlu{3.4}{-0.4}{\textup{\normalsize $G_{n,3}$, $n$ odd}}
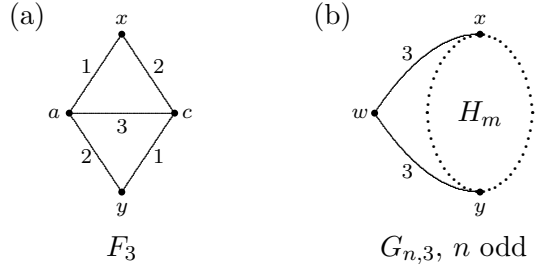
\ptlu{0}{-1.4}{\textup{\normalsize Figure 4. The graphs $F_3$ and $G_{n,3}$}}
\]\\[-1ex]
\indent(d) We first prove the two upper bounds. For $n\ge 4$, we construct a graph $G_1$ on $n$ vertices, with $|E(G_1)|=n+1$ and $rc_2(G_1)\le n-1$, as follows. We take the cycle $C_{n-1}=v_1v_2\cdots v_{n-1}v_1$, and connect another vertex $x$ to $v_1$ and $v_2$. Now, define the $(n-1)$-colouring $c_1$ on $G_1$,  where $c_1(v_iv_{i+1})=i$ for $1\le i\le n-1$ (with $v_n=v_1$), $c_1(xv_1)=2$, and $c_1(xv_2)=n-1$. Then $|E(G_1)|=n+1$, and it is easy to check that $c_1$ is a rainbow $2$-connected colouring of $G_1$, so that $rc_2(G_1)\le n-1$. It follows that $t_2(n,n-1)\le n+1$.

Next, let $n\ge 6$. We construct a graph $G_2$ on $n$ vertices, with $|E(G_2)|=n+2$ and $rc_2(G_2)\le n-2$, as follows. We take the cycle $C_{n-2}=v_1v_2\cdots v_{n-2}v_1$, and connect a vertex $x$ to $v_1$ and $v_2$, and a vertex $y$ to $v_2$ and $v_3$. Now, define the $(n-2)$-colouring $c_2$ on $G_2$,  where $c_2(v_iv_{i+1})=i$ for $1\le i\le n-2$ (with $v_{n-1}=v_1$), $c_2(xv_1)=2$, $c_2(xv_2)=n-2$, $c_2(yv_2)=3$, and $c_2(yv_3)=1$. Then $|E(G_2)|=n+2$, and it is easy to check that $c_2$ is a rainbow $2$-connected colouring of $G_2$, so that $rc_2(G_2)\le n-2$. It follows that $t_2(n,n-2)\le n+2$.

Now we prove the two lower bounds. First, let $n\ge 4$. Recall that the cycle $C_n$ is the unique $2$-connected graph on $n$ vertices with the minimum of edges, which is $n$. However, $rc_2(C_n)=n$. It follows that $t_2(n,n-1)\ge n+1$.

Next, let $n\ge 6$. By (\ref{t2chain}) and the above, we have $t_2(n,n-2)\ge t_2(n,n-1)\ge n+1$. Suppose that $G$ is a $2$-connected graph on $n$ vertices, and $|E(G)|=n+1$. We prove that $rc_2(G)\ge n-1$, and this implies the lower bound $t_2(n,n-2)\ge n+2$. By the remark after Theorem \ref{EDthm}, $G$ must be a cycle with one ear attached. Thus $G$ is a $\Theta$-graph, i.e., $G$ consists of three disjoint paths connecting two vertices, say $x$ and $y$. Let $Q_1,Q_2,Q_3$ be the three paths, on $q_1,q_2\ge 3$ and $q_3\ge 2$ vertices. Let $x_i,y_i\in V(Q_i)$ be the neighbours of $x$ and $y$ for $i=1,2$, and for $i=3$ if $q_3\ge 3$. Assume that there exists a rainbow $2$-connected colouring of $G$ with at most $n-2$ colours. In the two claims and the subsequent argument below, we obtain a contradiction by finding two vertices which are not connected by two disjoint rainbow paths. Call such a pair of vertices \emph{bad}. For convenience, whenever we state an edge $wz\in E(Q_i)$ below, we have $d_{Q_i}(x,w)=d_{Q_i}(x,z)-1$.

\begin{clm}\label{clmA}
If there exist two edges $ab\in E(Q_i)$ and $cd\in E(Q_j)$ with the same colour, for some $i\neq j$ with $q_i,q_j\ge 3$, then either $a=x$ and $d=y$, or $b=y$ and $c=x$.
\end{clm}
\begin{proof}
Suppose first that $a,b$ are internal vertices of $Q_i$. Then $\{a,c\}$ (resp.~$\{b,d\}$) is bad if $c$ (resp.~$d$) is an internal vertex of $Q_j$. Now suppose that $a=x$. If $d\neq y$, then $\{b,d\}$ is bad. Similarly, if $b=y$, then we have $c=x$.
\end{proof}


\begin{clm}\label{clmC}
There do not exist three edges with the same colour.
\end{clm}
\begin{proof}
Suppose that there exist three edges using the same colour, say red. If these red edges occur in all three paths, then we have a contradiction to Claim \ref{clmA}, unless (without loss of generality) $q_3=2$ and $xx_1,y_2y,xy$ are all red. But then $\{x_1,y\}$ is bad. Thus, some $Q_i$ contains two red edges. Lemma \ref{lm1} implies that $q_i=3$, so there is a red edge in $Q_j$ for some $j\neq i$. Claim \ref{clmA} then implies $q_j=2$, so that $j=3$, and we may assume that $i=2$. Claim \ref{clmA} again implies that there cannot exist a red edge in $Q_1$. Since $|E(G)|=n+1$ and we have at most $n-2$ colours, $Q_1$ must have two edges with the same colour other than red. But then $\{x,y\}$ is bad.
\end{proof}

Thus, since $|E(G)|=n+1$ and we have at most $n-2$ colours, Claim \ref{clmC} implies that there exist three colours, where for each colour, there are exactly two edges using the colour. Let red, blue and green be these three colours. \\[1ex]
\emph{Case 1.} $q_3=2$.\\[1ex]
\indent By Claim \ref{clmA}, it follows that $Q_3=xy$ must be one of the three colours, say green. If say $Q_2$ contains both blue edges, then Lemma \ref{lm1} implies $q_2=3$. But then both red edges occur in $Q_1$, and $\{x,y\}$ is bad. Thus by Claim \ref{clmA}, we may assume that $xx_1, y_2y$ are red, and $xx_2,y_1y$ are blue. Now, the remaining green edge is in $Q_i$ for some $i\in\{1,2\}$, and $\{x,x_i\}$ is bad.\\[1ex]
\emph{Case 2.} $q_3\ge 3$.\\[1ex]
\indent Suppose, say, that the two green edges are in $Q_3$. By Lemma \ref{lm1}, $q_3=3$, and $xx_3,x_3y$ are green. If, say, the two blue edges are in $Q_2$, then Lemma \ref{lm1} implies that $q_2=3$, and $xx_2,x_2y$ are blue. But then $\{x,y\}$ is bad. Otherwise, Claim \ref{clmA} implies that, we may assume $xx_1, y_2y$ are red, and $xx_2,y_1y$ are blue. But then $\{x,x_1\}$ is bad.

Hence, none of the three paths may contain two edges in red, blue, or green. Now by Claim \ref{clmA}, we may assume that $xx_1,y_2y$ are red, $xx_2,y_3y$ are blue, and $xx_3,y_1y$ are green. But then, since $n\ge 6$, we have $q_i\ge 4$ for some $i$, and $\{x_i,y_i\}$ is bad.

Therefore, we have $rc_2(G)\ge n-1$, as required.\\[1ex]
\indent This completes the proof of Theorem \ref{t2(n,r)thm}.
\end{proof}


\section{The function $s_2(n,r)$}\label{s2(n,r)sect}

Let $n>k\ge 2$. If $n\ge (k+1)^2$, then by Theorem \ref{CJMZthm}(b), we have $s_k(n,2)={n\choose 2}$. Now, define
\[
M_{n,k}=\max\{r:rc_k(G)=r\textup{ for some }k\textup{-connected graph $G$ on $n$ vertices$\}$.}
\]
Then $s_k(n,r)$ is well-defined if and only if $2\le r\le M_{n,k}$. Together with Theorem \ref{Harthm}, we obtain the following analogue to (\ref{s1chain})
\[
{n\choose 2} = s_k(n,2)\ge s_k(n,3)\ge\cdots\ge s_k(n,M_{n,k})\ge\Big\lceil\frac{kn}{2}\Big\rceil.
\]

Now, we focus on the function $s_2(n,r)$. Let $n\ge 4$. Theorem \ref{CJMZthm}(a) implies that $s_2(n,2)={n\choose 2}$. Theorem \ref{LLthm} then implies that $s_2(n,r)$ is well-defined if and only if $2\le r\le n$. Moreover, if $G$ is a $2$-connected graph on $n$ vertices with $|E(G)|\ge n+1$, then $G\neq C_n$, so Theorem \ref{LLthm} gives $rc_2(G)\le n-1$. This implies that $s_2(n,n)=n$. Therefore, for $n\ge 4$,
\[
{n\choose 2} = s_2(n,2)\ge s_2(n,3)\ge\cdots\ge s_2(n,n)=n.
\]
Thus, it remains to consider $3\le r\le n-1$. We provide lower bounds for $s_2(n,r)$ in the following proposition.

\begin{prop}\label{s2(n,r)prop}
\indent\\[-3ex]
\begin{enumerate}
\item[(a)] Let $3\le r\le n-1$. Then $s_2(n,r)\ge {n-r+2 \choose 2}+r-1$.
\item[(b)] Let $n\ge 6$, and $\frac{n}{2}+2\le r\le n-1$. Then $s_2(n,r)\ge {n-r+3\choose 2}+r-3$.
\end{enumerate}
\end{prop}


\begin{proof}
(a) Let $G$ be the graph on $n$ vertices, obtained by taking the complete graph $K_{n-r+2}$ and attaching an $(r-1)$-ear $P$ at two vertices $v_0,v_{r-1}$. Then $|E(G)|={n-r+2 \choose 2}+r-1$. We show that $rc_2(G)\ge r$, which implies that $s_2(n,r)\ge {n-r+2 \choose 2}+r-1$.

Let $P=v_0v_1\cdots v_{r-1}$ and $u_1,\dots,u_{n-r}$ be the remaining vertices of $G$. Suppose that there exists a rainbow $2$-connected colouring $c$ of $G$, using at most $r-1$ colours. The two disjoint rainbow paths connecting $v_{r-1}$ and $v_{r-2}$ must be $v_{r-1}v_{r-2}$ and $v_{r-1}v_0v_1\cdots v_{r-2}$, so we may assume that $c(v_{i-1}v_i)=i$ for $1\le i\le r-2$ and $c(v_0v_{r-1})=r-1$. Similarly, by considering $v_0$ and $v_1$, the path $v_1v_2\cdots v_{r-1}v_0$ must be rainbow, and so $c(v_{r-2}v_{r-1})=1$. Now for any $1\le j\le n-r$, to connect $v_1$ and $u_j$ with two disjoint rainbow paths, one path must be $v_1v_2\cdots v_{r-1}u_j$. Thus $c(v_{r-1}u_j)=r-1$. Similarly, by considering $v_{r-2}$ and $u_j$, we have $c(v_0u_j)=r-1$. But now, we see that there do not exist two disjoint rainbow paths connecting $v_0$ and $v_{r-1}$. Therefore, $rc_2(G)\ge r$.\\[1ex]
\indent (b) Note that we have $r\ge 5$. We construct the graph $G$ on $n\ge 6$ vertices as follows. Take the complete graph $K_{n-r+3}$ and delete one edge $v_0v_{r-2}$. Then add an $(r-2)$-ear $P$ at $v_0$ and $v_{r-2}$. We have $|E(G)|= {n-r+3\choose 2}+r-3$. We show that $rc_2(G)\ge r$, which implies that $s_2(n,r)\ge {n-r+3\choose 2}+r-3$.

Let $P=v_0v_1\cdots v_{r-2}$ and $u_1,\dots,u_{n-r+1}$ be the remaining vertices of $G$. Suppose that there exists a rainbow $2$-connected colouring $c$ of $G$, using at most $r-1$ colours. Since $P$ has length $r-2\ge 3$, Lemma \ref{lm1} implies that $P$ must be rainbow coloured. We may assume that $c(v_{i-1}v_i)=i$ for $1\le i\le r-2$. Now for $1\le i\le r-2$, to connect $v_{i-1}$ and $v_i$ with two disjoint rainbow paths, one path must be $v_{i-1}Pv_0u_jv_{r-2}Pv_i$, for some $1\le j\le n-r+1$. Note that $n-r+2\le r-2$, so that the vertices $v_0,v_1,\dots,v_{n-r+2}$ all exist. By considering the vertices $v_{i-1}$ and $v_i$ for $1\le i\le n-r+1$, we may assume that for $1\le i\le n-r+1$, the path $v_0u_iv_{r-2}$ uses the colours $i$ and $r-1$. But now, we see that there do not exist  two disjoint rainbow paths connecting $v_{n-r+1}$ and $v_{n-r+2}$. Therefore, $rc_2(G)\ge r$.
\end{proof}


For $3\le r\le n-1$, we believe that the lower bounds in Proposition \ref{s2(n,r)prop} are close to the correct values of $s_2(n,r)$. We propose the following problem.

\begin{pro}
For $3\le r\le n-1$, determine the function $s_2(n,r)$. Do there exist $n$ and $r$ such that $s_2(n,r)$ is equal to either lower bound in Proposition \ref{s2(n,r)prop}?
\end{pro}


\section*{Acknowledgements}
Shinya Fujita is supported by JSPS KAKENHI (No.~19K03603). Henry Liu is partially supported by the Startup Fund of One Hundred Talent Program of SYSU, and National Natural Science Foundation of China (No.~11931002). Boram Park is supported by Basic Science Research Program through the National Research Foundation of Korea (NRF) funded by the Ministry of Science, ICT and Future Planning (NRF-2018R1C1B6003577).

Shinya\, Fujita\, and\, Boram\, Park\, acknowledge\, the\, generous\, hospitality\, of\, Sun\, Yat-sen University, Guangzhou, China. They were able to carry out part of this research with Henry Liu during their visits there in 2019.

The\, authors\, would\, like\, to\, thank\, the\, anonymous\, referee\, for\, valuable\, comments\, and suggestions.


\begin{thebibliography}{99}
\bibitem{BCLH19}X.~Bai, R.~Chang, Z.~Huang, X.~Li, More on the rainbow disconnection in graphs, Discuss.~Math.~Graph Theory, in press.

\bibitem{BL19} X.~Bai, X.~Li, Erd\H{o}s-Gallai-type results for the rainbow disconnection number of graphs, ArXiv preprint arXiv:1901.02740.

\bibitem{BH13} J-P.~Bode, H.~Harborth, The minimum size of $k$-rainbow connected graphs of given order, Discrete Math.~313 (2013) 1924--1928.

\bibitem{Bol78} B.~Bollob\'{a}s, Extremal Graph Theory, Academic Press, London, 1978.

\bibitem{Bol98} B. Bollob\'as, Modern Graph Theory, Springer-Verlag, New York, 1998.

\bibitem{CLW17} Q.~Cai, X.~Li, D.~Wu, Erd\H{o}s-Gallai-type results for colorful monochromatic connectivity of a graph, J.~Comb.~Optim.~33 (2017) 123--131.

\bibitem{CDHHZ18} G.~Chartrand,\: S.~Devereaux,\: T.W.~Haynes,\: S.T.~Hedetniemi,\: P.~Zhang,\: Rainbow disconnection in graphs, Discuss.~Math.~Graph Theory 38 (2018) 1007--1021.

\bibitem{CJMZ08} G.~Chartrand, G.L.~Johns, K.A.~McKeon, P.~Zhang, Rainbow connection in graphs, Math. Bohem.~133 (2008) 85--98.

\bibitem{CJMZ09} G.~Chartrand, G.L.~Johns, K.A.~McKeon, P.~Zhang, The rainbow connectivity of a graph, Networks 54 (2009) 75--81.

\bibitem{Die18} R.~Diestel, Graph Theory, 5th ed., Springer-Verlag, New York, 2018.

\bibitem{Har62} F.~Harary, The maximum connectivity of a graph, Proc.~Nat.~Acad.~Sci.~USA 48 (1962) 1142--1146.

\bibitem{KPSW13} A.~Kemnitz, J.~Przyby\l{}o, I.~Schiermeyer, M.~Wo\'zniak, Rainbow connection in sparse graphs, Discuss.~Math.~Graph Theory 33 (2013) 181--192.

\bibitem{KS11} A.~Kemnitz, I.~Schiermeyer, Graphs with rainbow connection number two, Discuss. Math.~Graph Theory 31 (2011) 313--320.

\bibitem{LLSZ14} H.~Li, X.~Li, Y.~Sun, Y.~Zhao, Note on minimally $d$-rainbow connected graphs, Graphs Combin.~30 (2014) 949--955.

\bibitem{LLS13} X.~Li, M.~Liu, I.~Schiermeyer, Rainbow connection number of dense graphs, Discuss. Math.~Graph Theory 33 (2013) 603--611.

\bibitem{LL13} X.~Li, S.~Liu, A sharp upper bound for the rainbow 2-connection number of a 2-connected graph, Discrete Math.~313 (2013) 755--759.

\bibitem{LSS13} X.~Li, Y.~Shi, Y.~Sun, Rainbow connections of graphs: a survey, Graphs Combin.~29  (2013) 1--38.

\bibitem{LS12} X.~Li, Y.~Sun, Rainbow Connections of Graphs, Springer-Verlag, New York, 2012.

\bibitem{Lo14} A.~Lo, A note on the minimum size of $k$-rainbow-connected graphs, Discrete Math.~331 (2014) 20--21.

\bibitem{Men27} K.~Menger, Zur allgemeinen Kurventheorie, Fund.~Math.~10 (1927) 96--115.

\bibitem{Sch13} I.~Schiermeyer, On minimally rainbow $k$-connected graphs, Discrete Appl.~Math.~161 (2013) 702--705.

\end{thebibliography}
\end{document}